\documentclass[reqno,12pt,a4paper]{amsart}

\usepackage{enumerate}
\usepackage{amstext,amsmath,amsthm,amsfonts,amssymb,amscd}
\usepackage{latexsym,mathrsfs,dsfont,euscript}
\usepackage{pict2e}
\usepackage{fullpage}
\usepackage{xcolor}%color

\usepackage{hyperref} % hace links referencias a todos los labels
\hypersetup{
  colorlinks,%
  citecolor=blue,%
    filecolor=red,%
    linkcolor=red,%
    urlcolor=blue
}

\usepackage{tikz}
\newtheorem{theorem}{Theorem}[section]
\newtheorem{proposition}[theorem]{Proposition}
\newtheorem{lemma}[theorem]{Lemma}

\theoremstyle{definition}
\newtheorem{definition}[theorem]{Definition}

\theoremstyle{remark}

\numberwithin{equation}{section}

\DeclareMathOperator{\supp}{supp}
\newcommand{\abs}[1]{\left\vert#1\right\vert}

\newcommand{\proin}[2]{\left<#1,#2\right>}
\newcommand{\norm}[1]{\left\Vert#1\right\Vert}

\allowdisplaybreaks
\begin{document}
\title[]{Partial derivatives, singular integrals and Sobolev Spaces in dyadic settings}
%

%%    Information for authors

\author[]{Hugo Aimar}
%\email{haimar@santafe-conicet.gov.ar}
%
\author[]{Juan Comesatti}
%\email{jcomesatti@santafe-conicet.gov.ar}
%
\author[]{Ivana G\'{o}mez}
%\email{ivanagomez@santafe-conicet.gov.ar}
%
\author[]{Luis Nowak}
%\email{luis.nowak@faea.uncoma.edu.ar}
%%

%%
\thanks{This work was supported by the MINCYT in Argentina: CONICET and ANPCyT; UNL and UNComa}
% 	42C40   	Wavelets and other special systems
%	26A33   	Fractional derivatives and integrals


\begin{abstract}
In this note we show that the general theory of vector valued singular integral operators of Calder\'on-Zygmund defined on general metric measure spaces, can be applied to obtain Sobolev type regularity properties for solutions of the dyadic fractional Laplacian. In doing so, we define partial derivatives in terms of Haar multipliers and dyadic homogeneous singular integral operators.
\end{abstract}
\maketitle

\section{Introduction}\label{sec:introduction}

In order to properly state the underlying ideas in the main problems and results of this paper, let us start by a well known and classical result that we give in a form which is suitable for our setting. Set $\mathscr{S}(\mathbb{R}^n)$ to denote the class of Schwartz test functions of the theory of distributions in the Euclidean space $\mathbb{R}^n$. As usual we shall use $\widehat{f}$ to denote the Fourier transform of $f$. A linear operator $L: \mathscr{S}(\mathbb{R}^n)\to \mathscr{S}(\mathbb{R}^n)$ is a second order differential operator in $\mathbb{R}^n$ of the form $L\varphi =\sum_{i,j=1}^n a_{ij}\frac{\partial^2\varphi}{\partial x_i\partial x_j}$ if and only if $\widehat{L\varphi}(\xi)=m(\xi)\widehat{\Delta\varphi}(\xi)$ with $m(\xi)$ homogeneous of degree zero and $m(\xi)$ on the unit sphere $S^{n-1}$ of $\mathbb{R}^n$ is a second order polynomial.
The proof is just  the application of the basic rules for the Fourier Transform of derivatives. Hence, in a sense, the family of all constant coefficient second order differential operators can be identified with the class $\mathscr{M}$ of those special multipliers  $m(\xi)$.

A lower order instance of the above remark is provided by first order derivatives defined in terms of the operator $\sqrt{-\Delta}$. In fact, $\widehat{\frac{\partial}{\partial x_j}}=c\frac{\xi_j}{\abs{\xi}}(\sqrt{-\Delta})^{\widehat{}}$. In other words, the partial derivatives can be regarded as the composition of $\sqrt{-\Delta}$ and a Calder\'on-Zygmund singular integral, actually the $j$-th Riesz transform whose Fourier multiplier is $\frac{\xi_j}{\abs{\xi}}$. Aside from the relevance of this interplay between the Fourier Transform and differential operators in the theory of Calder\'on-Zygmund of Sobolev regularity, sometimes a Laplacian type operator and a particular Fourier type analysis are known, nevertheless there is no a natural way to understand partial or directional derivatives. The above identification of linear PDE with $\mathscr{M}$ provides an analytical natural way to define differential operators associated with the Laplacian of the setting. In this note we consider these problem in the dyadic setting. For simplicity we shall work our theory in $\mathbb{R}^+=[0,\infty)$, since it becomes a quadrant for the standard dyadic intervals in $\mathbb{R}$. The theory extends naturally to more general underlying spaces and more general dyadic families. Our program includes the analysis of the homogeneity properties of kernels and multipliers, the definition of partial derivatives and gradient and the characterization through norms of the gradient of the Sobolev spaces provided by the dyadic fractional Laplacian. The main tool will be a vector valued theorem of singular integrals of Calder\'on-Zygmund type defined on a space of homogeneous type. Following ideas in \cite{AiGoPetermichl} we give a model in an adequate space of homogeneous type of the dyadic setting, now with values in an infinite dimensional sequence space of the type $\ell^2$.

The paper is organized as follows. In Section~\ref{sec:DyadicSetting} we describe the basic dyadic setting, including some partitions of the class of all dyadic intervals in subfamilies that will play a special role in the analysis of homogeneity properties of the kernels that we accomplish in Section~\ref{sec:HomogeneousMultipliers}. Section~\ref{sec:LaplacianGradient} is devoted to introduce directional and partial derivatives and to show that the $L^2$ norms of the gradient of $f$ is the right dyadic fractional energy contained in $f$. We also obtain the formulas, via singular integrals and the fractional dyadic Laplacian, for directional and partial derivatives and for the gradient in terms of a vector valued singular integral. In Section~\ref{sec:CalderonZygmundSingularETH} we briefly sketch the vector valued theory of singular integrals on spaces of homogeneous type that we use in Section~\ref{sec:SobolevRegularity} in order to prove our main result showing the boundedness of the $L^p$ norm of the gradient of $f$ in terms of the $L^p$ norm of the dyadic fractional Laplacian $(-\Delta)^{s/2}_{dy}$.

\section{The dyadic system in $\mathbb{R}^+$. Dyadic metric. Haar basis}\label{sec:DyadicSetting}

In the statements and proofs of all the definitions and results of this paper the notation will play an important role at simplifying the presentation. In this section we introduce the basic notations and known facts in $\mathbb{R}^+$ in order to review the concepts, but mostly in order to establish the notational agreements. For $j\in\mathbb{Z}$ and $k\geq 0$, a nonnegative integer, set $I^j_k=[k2^{-j},(k+1)2^{-j})$. For $j\in\mathbb{Z}$, $\mathcal{D}^j=\{I^j_k: k\geq 0\}$, $\mathcal{D}=\cup_{j\in\mathbb{Z}}\mathcal{D}^j$ is the family of all dyadic intervals in $\mathbb{R}^+$. For $x$, $y\in\mathbb{R}^+$ with $x\neq y$ set $I(x,y)$ to denote the smallest interval in $\mathcal{D}$ such that $x$ and $y$, both, belong to $I(x,y)$. The minimality of $I(x,y)$ guarantees that $x$ and $y$ belong to different dyadic halves of $I(x,y)$. The measure $\abs{I(x,y)}=\delta(x,y)$ provides  metric (ultra-metric) in $\mathbb{R}^+$ which is larger than the Euclidean but not equivalent. The $\delta$-balls are the dyadic intervals. In our further analysis we shall use some well known basic properties of the normal (or $1$-Ahlfors) space of homogeneous type structure induced in $\mathbb{R}^+$ by $\delta$ and the standard Lebesgue measure. In particular we shall use that the singular behavior in the sense of integrals of the powers of $\delta(x,y)$ are exactly those of the corresponding powers of $\abs{x-y}$. Hence $(\delta(x,y))^{-2}$ is integrable as a function of $y\neq B_{\delta}(x,\varepsilon)$ uniformly in $x$. The function $h^0_0(x)=\mathcal{X}_{[0,\tfrac{1}{2})}(x)-\mathcal{X}_{[\tfrac{1}{2},1)}(x)$ is the basic Haar wavelet associated to, and supported in, $I^0_0=[0,1)\in\mathcal{D}$. The scaling $h^j_k(x)=2^{j/2}h^0_0(2^jx-k)$ provides the Haar wavelets supported in $I=I^j_k\in\mathcal{D}$. The sequence of all these wavelets, $\mathscr{H}$, is an orthonormal basis for the space $L^2(\mathbb{R}^+,dx)$. Actually, the above construction determines a bijection between $\mathscr{H}$ and $\mathcal{D}$. So, sometimes we shall use $I(h)$ to denote the dyadic interval supporting $h$ and sometimes we shall write $h_I$ to denote the Haar wavelet supported in $I$. Also for every $(j,k)\in\mathbb{Z}\times\mathbb{Z}^+$ we associate a dyadic interval $I^j_k$ and a Haar wavelet $h^j_k$.

Since $\mathbb{R}^+$ is still a cone in the algebraic sense, we can multiply the points in $\mathbb{R}^+$ by positive numbers and we remain in $\mathbb{R}^+$. Hence for a given set $A\subset \mathbb{R}^+$ and a given positive number, as usual, $\alpha A=\{\alpha a: a \in A\}\subset \mathbb{R}^+$. In particular, $2I=[2a,2b)$ of $I=[a,b)$ with $0\leq a<b$. An important fact in the further development is that the operation of multiplication by $2$ preserves $\mathcal{D}$. Moreover it is a bijection of $\mathcal{D}$. Being $\tfrac{1}{2}I$ its inverse.

On the other hand, there exists another notion of duplication of dyadic intervals that is given by ancestry. At some points of our analysis we shall need a simple notation to denote ancestry. For $I\in\mathcal{D}$ we shall write $I^{(1)}$ in order to denote the first ancestor of $I$. In other words, if $I\in\mathcal{D}^j$, $I^{(1)}$ is the only $J\in\mathcal{D}^{j-1}$ such that $J\supset I$. In general, we shall write $I^{(l)}$, $l\geq 1$, to denote the $l$-th ancestor of $I$.

The two notions of \textit{duplication} of $I$, provided by $2I$ and $I^{(1)}$ generally do not coincide. Actually they coincide only when $I=[0,2^{-j})$ for some $j\in\mathbb{Z}$.

The sequence of nested partitions provided by $\mathcal{D}$ on $\mathbb{R}^+$ induces a partition of ${\mathbb{R}^+}^2=\mathbb{R}^+\times \mathbb{R}^+$ in terms of the level sets of the metric $\delta$. In fact, for $I\in\mathcal{D}$ set $B(I)=(I^+\times I^-)\cup (I^-\times I^+)$ and $\Lambda_j=\cup_{I\in\mathcal{D}_j}B(I)$, $j\in\mathbb{Z}$. Then the sets $\Lambda_j$ are pairwise disjoint and $\mathbb{R}^+\times \mathbb{R}^+=\cup_{j\in\mathbb{Z}}\Lambda_j\cup\triangle$, where $\triangle$ is the diagonal of $\mathbb{R}^+\times \mathbb{R}^+$. Moreover, $\Lambda_j=\{(x,y):\delta(x,y)=2^{-j}\}$ for every $j\in\mathbb{Z}$. Set $\mathscr{B}=\{B=B(I):I\in\mathcal{D}\}$ to denote the set of all butterflies. A partition of $\mathscr{B}$ is induced by the equivalence relation $B_1\sim B_2$ if and only if there exists an integer $l\in\mathbb{Z}$ such that $B_2=B(2^lI)$ if $B_1=B(I)$. Since for $I=I^j_k=[k2^{-j},(k+1)2^{-j})$ we have that $2^lI=2^lI^j_k=[k2^{l-j},(k+1)2^{l-j})$, then for every $I=I^j_k$ there exists one and only $J$ in $\mathcal{D}^0$ with $B(I)\sim B(J)$ (namely $J=I^0_k$). Hence
\begin{equation*}
\mathscr{B}=\bigcup_{k\geq 0} \widetilde{B}([k,k+1))
\end{equation*}
where $\widetilde{B}([k,k+1))=\{B\in\mathscr{B}: B\sim B([k,k+1))\}$ is the equivalence class of $B([k,k+1))$. Since
\begin{equation*}
(\mathbb{R}^+)^2\setminus \triangle = \bigcup_{B\in\mathscr{B}} B= \bigcup_{k\geq 0}\bigcup_{B\in \widetilde{B}([k,k+1))} B = \bigcup_{k\geq 0} \Gamma_k
\end{equation*}
with $\Gamma_k = \bigcup_{B\in\widetilde{B}([k,k+1))}B$. See Figure~\ref{fig:LevelSetsdelta_and_Gamma0}.

%%%%%%%%%%%%%%%%%%%%%%%%%%%%%%%%%%%%%%%%%%%%%%%%%%%%%%%%%%%%%%%%%%
%%%%%%%%%% Figure of Level Sets (Butterflies) and Gamma_k %%%%%%%%
%%%%%%%%%%%%%%%%%%%%%%%%%%%%%%%%%%%%%%%%%%%%%%%%%%%%%%%%%%%%%%%%%%

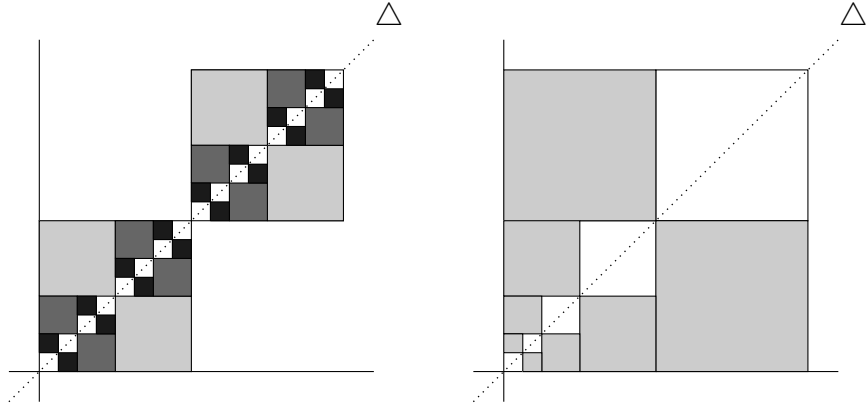
\begin{figure}[ht]
	%\centering
	%\begin{minipage}{0.5\textwidth}
		%\centering
	%\begin{center}
		\begin{tikzpicture}[scale=2]
		\begin{minipage}{0.5\textwidth}
		\foreach \x in {0,1}%{0,...,2}
		\draw[-] (0+\x,0+\x)--(1+\x,0+\x)--(1+\x,1+\x)--(0+\x,1+\x)--(0+\x,0+\x);
		%		\draw[-] (0+\x,0+\x)--(1+\x,0+\x)--(1+\x,1+\x)--(0+\x,1+\x)--(0+\x,0+\x);
		%
		\draw[-] (0,0)--(-.2,0);
		\draw[-] (0,0)--(0,-.2);
		%
		%\draw[-] (2,2)--(2.2,2);
		%\draw[-] (2,2)--(2,2.2);
		%
		\draw[-] (1,0)--(2.2,0);
		\draw[-] (0,1)--(0,2.2);
		\draw[line width=.5pt, dotted] (-.2,-0.2)--(2.2,2.2);
		\draw (2.3,2.3)node[] {\scalebox{1}{$\displaystyle \triangle$}};
		\foreach \x in {0,1}%{0,...,2}
		\draw[fill=lightgray!80] (.5+\x,0+\x)--(1+\x,0+\x)--(1+\x,.5+\x)--(.5+\x,.5+\x)--(.5+\x,0+\x);
		\foreach \x in {0,1}%{0,...,2}
		\draw[fill=lightgray!80] (0+\x,0.5+\x)--(.5+\x,.5+\x)--(.5+\x,1+\x)--(0+\x,1+\x)--(0+\x,.5+\x);
		\foreach \x in {0,1}%{0,...,2}
		\draw[fill=black!60] (.25+\x,0+\x)--(.5+\x,0+\x)--(.5+\x,.25+\x)--(.25+\x,.25+\x)--(.25+\x,0+\x);
		\foreach \x in {0,1}%{0,...,2}
		\draw[fill=black!60] (0+\x,.25+\x)--(.25+\x,.25+\x)--(.25+\x,.5+\x)--(0+\x,.5+\x)--(0+\x,.25+\x);
		\foreach \x in {0,1}%{0,...,2}
		\draw[fill=black!60] (0.75+\x,.5+\x)--(1+\x,.5+\x)--(1+\x,.75+\x)--(.75+\x,.75+\x)--(0.75+\x,.5+\x);
		\foreach \x in {0,1}%{0,...,2}
		\draw[fill=black!60] (0.5+\x,.75+\x)--(.75+\x,.75+\x)--(.75+\x,1+\x)--(.5+\x,1+\x)--(0.5+\x,.75+\x);
		\foreach \x in {0,1/2,1,1.5}%{0,...,2}
		\draw[fill=black!90] (.25/2+\x,0+\x)--(.5/2+\x,0+\x)--(.5/2+\x,.25/2+\x)--(.25/2+\x,.25/2+\x)--(.25/2+\x,0+\x);
		\foreach \x in {0,1/2,1,1.5}%{0,...,2}
		\draw[fill=black!90] (0+\x,.25/2+\x)--(.25/2+\x,.25/2+\x)--(.25/2+\x,.5/2+\x)--(0+\x,.5/2+\x)--(0+\x,.25/2+\x);
		\foreach \x in {0,1/2,1,1.5}%{0,...,2}
		\draw[fill=black!90] (0.75/2+\x,.5/2+\x)--(1/2+\x,.5/2+\x)--(1/2+\x,.75/2+\x)--(.75/2+\x,.75/2+\x)--(0.75/2+\x,.5/2+\x);
		\foreach \x in {0,1/2,1,1.5}%{0,...,2}
		\draw[fill=black!90] (0.5/2+\x,.75/2+\x)--(.75/2+\x,.75/2+\x)--(.75/2+\x,1/2+\x)--(.5/2+\x,1/2+\x)--(0.5/2+\x,.75/2+\x);
		\end{minipage}
		\end{tikzpicture}
	%\end{center}
%\end{minipage}
\hspace{5mm}
%\begin{minipage}{0.5\textwidth}
%	\centering
%\begin{center}
	\begin{tikzpicture}[scale=2]
	\begin{minipage}{0.5\textwidth}
	\foreach \x in {0,1}%{0,...,2}
	\draw[-] (0+\x,0+\x)--(1+\x,0+\x)--(1+\x,1+\x)--(0+\x,1+\x)--(0+\x,0+\x);
	%		\draw[-] (0+\x,0+\x)--(1+\x,0+\x)--(1+\x,1+\x)--(0+\x,1+\x)--(0+\x,0+\x);
	%
	\draw[-] (0,0)--(-.2,0);
	\draw[-] (0,0)--(0,-.2);
	%
	%\draw[-] (2,2)--(2.2,2);
	%\draw[-] (2,2)--(2,2.2);
	%
	\draw[-] (1,0)--(2.2,0);
	\draw[-] (0,1)--(0,2.2);
	\draw[line width=.5pt, dotted] (-.2,-0.2)--(2.2,2.2);
	\draw (2.3,2.3)node[] {\scalebox{1}{$\displaystyle \triangle$}};
	\foreach \x in {0}%{0,...,2}
	\draw[fill=lightgray!80] (1+\x,0+\x)--(2+\x,0+\x)--(2+\x,1+\x)--(1+\x,1+\x)--(1+\x,0+\x);
	\foreach \x in {0}%{0,...,2}
	\draw[fill=lightgray!80] (0+\x,1+\x)--(1+\x,1+\x)--(1+\x,2+\x)--(0+\x,2+\x)--(0+\x,1+\x);
	\foreach \x in {0}%{0,...,2}
	\draw[fill=lightgray!80] (.5+\x,0+\x)--(1+\x,0+\x)--(1+\x,.5+\x)--(.5+\x,.5+\x)--(.5+\x,0+\x);
	\foreach \x in {0}%{0,...,2}
	\draw[fill=lightgray!80] (0+\x,0.5+\x)--(.5+\x,.5+\x)--(.5+\x,1+\x)--(0+\x,1+\x)--(0+\x,.5+\x);
	\foreach \x in {0}%{0,...,2}
	\draw[fill=lightgray!80] (.25+\x,0+\x)--(.5+\x,0+\x)--(.5+\x,.25+\x)--(.25+\x,.25+\x)--(.25+\x,0+\x);
	\foreach \x in {0}%{0,...,2}
	\draw[fill=lightgray!80] (0+\x,.25+\x)--(.25+\x,.25+\x)--(.25+\x,.5+\x)--(0+\x,.5+\x)--(0+\x,.25+\x);
	%		\foreach \x in {0,1}%{0,...,2}
	%		\draw[fill=black!60] (0.75+\x,.5+\x)--(1+\x,.5+\x)--(1+\x,.75+\x)--(.75+\x,.75+\x)--(0.75+\x,.5+\x);
	%		\foreach \x in {0,1}%{0,...,2}
	%		\draw[fill=black!60] (0.5+\x,.75+\x)--(.75+\x,.75+\x)--(.75+\x,1+\x)--(.5+\x,1+\x)--(0.5+\x,.75+\x);
	%		%%
	\foreach \x in {0}%{0,...,2}
	\draw[fill=lightgray!80] (.25/2+\x,0+\x)--(.5/2+\x,0+\x)--(.5/2+\x,.25/2+\x)--(.25/2+\x,.25/2+\x)--(.25/2+\x,0+\x);
	\foreach \x in {0}%{0,...,2}
	\draw[fill=lightgray!80] (0+\x,.25/2+\x)--(.25/2+\x,.25/2+\x)--(.25/2+\x,.5/2+\x)--(0+\x,.5/2+\x)--(0+\x,.25/2+\x);
	%		\foreach \x in {0,1/2,1,1.5}%{0,...,2}
	%		\draw[fill=black!90] (0.75/2+\x,.5/2+\x)--(1/2+\x,.5/2+\x)--(1/2+\x,.75/2+\x)--(.75/2+\x,.75/2+\x)--(0.75/2+\x,.5/2+\x);
	%		\foreach \x in {0,1/2,1,1.5}%{0,...,2}
	%		\draw[fill=black!90] (0.5/2+\x,.75/2+\x)--(.75/2+\x,.75/2+\x)--(.75/2+\x,1/2+\x)--(.5/2+\x,1/2+\x)--(0.5/2+\x,.75/2+\x);
	%		%		
	\end{minipage}
	\end{tikzpicture}

\caption{Left; the level sets $\Lambda_0$ (lightgray), $\Lambda_1$ (darkgray) and $\Lambda_2$ (black) of $\delta$. Right; $\Gamma_0$. }\label{fig:LevelSetsdelta_and_Gamma0}
\end{figure}

\section{The class $\mathscr{M}$ of dyadic homogeneous multipliers for $\mathscr{H}$. The class $\mathcal{K}$ of dyadic homogeneous kernels on $(\mathbb{R}^+)^2$} \label{sec:HomogeneousMultipliers}

In this section we explore an analogous for the Haar analysis of the well known fact that the Fourier Transform $\widehat{K}$ of a convolution kernel $K$ homogeneous of degree $m$ is homogeneous of degree $-m-n$ if $n$ denotes the dimension of the underlying space.

In the Euclidean setting $\mathbb{R}^n$, a homogeneous function is completely determined by its values on the unit sphere $S^{n-1}$ and by the degree of homogeneity. The role of $S^{n-1}$ in our setting is performed by the set $\mathcal{D}^0$ of all the dyadic intervals with length equal to one, in a sense that we shall precise.

Let us start by the definition of homogeneous multipliers for the Haar system $\mathscr{H}$.
Let $m:\mathscr{H}\to\mathbb{R}$, be bounded, $\abs{m(h)}\leq\norm{m}_{\infty}<\infty$ for every $h\in\mathscr{H}$. We say that $m\in\mathscr{M}$, or that $m$ es homogeneous of degree zero if $m(I)=m(2I)$ for every $I\in\mathcal{D}$, (here $m(I):=m(h_I)$) and, with $I=I^j_k=[k2^{-j},(k+1)2^{-j})$, $2I=[2k2^{-j},2(k+1)2^{-j})=I^{j-1}_k$. The next result shows that a function $m\in\mathscr{M}$ is characterized by its values on the Haar functions (or the dyadic intervals) of level $j=0$.

\begin{proposition}
	Let $m:\mathcal{D}\to\mathbb{R}$, bounded. If $m\in\mathscr{M}$ then for $I^j_k\in\mathcal{D}$, we have	 $m(I^j_k)=m(I^0_k)$.
\end{proposition}
\begin{proof}
	Assume $j<0$, then $m(I^j_k)=m(2I^{j+1}_{k})=m(I^{j+1}_k)=\ldots=m(I^0_k)$.
\end{proof}

Now we introduce the notion of homogeneous kernels in the dyadic setting and we relate the two types of homogeneity via the Haar-Fourier analysis.
Let  $\mathscr{F}=\sigma(\mathscr{B})$ be $\sigma$-algebra of subsets of $(\mathbb{R}^+)^2$ generated by $\mathscr{B}$.

\begin{definition}
	An $\mathscr{F}$ measurable real function $K: (\mathbb{R}^+)^2\to \mathbb{R}$ is said to be dyadically homogeneous of degree zero if $K$ is constant on each $\Gamma_k$. We shall also write $K\in\mathcal{K}$.
\end{definition}
\begin{proposition}
	$K\in\mathcal{K}$ if and only if $K(2x,2y)=K(x,y)$ for every $x\neq y$.
\end{proposition}
\begin{proof}
	Follows from the fact that $(x,y)\in \Gamma_k$ if and only if $(2x,2y)\in \Gamma_k$.
\end{proof}
%Notice that the $\mathscr{F}$ measurability of $K$ means that $K$ is constant on each $B\in\mathscr{B}$. Set $\omega_k$ to denote the value of $K$ on $\Gamma_k$.

The next lemma provides a useful way of construction of homogeneous kernels of degree zero in terms of Haar functions and of homogeneous multipliers of degree zero.
\begin{lemma}\label{lem:OmegamDyadicallyHomogeneous}
	Let $m:\mathcal{D}\to\mathbb{R}$ be a bounded sequence indexed on the dyadic sets $\mathcal{D}$ of $\mathbb{R}^+$. Assume that $m(I)=m(2I)$ for every $I\in\mathcal{D}$. Then the function defined on $(\mathbb{R}^+)^2\setminus \triangle$ by
	\begin{equation*}
	\Omega_m(x,y)=\delta(x,y)\sum_{I\in\mathcal{D}} m(I) h_I(x) h_I(y)
	\end{equation*}
	is dyadically homogeneous of degree zero.
\end{lemma}
\begin{proof}
	Given an interval $J\in\mathcal{D}$ we shall write $J^{(1)}$ to denote the first ancestor of $J$. In general, for $j\geq 1$ write $J^{(j)}$ to denote the $j$-th ancestor of $J$. A caveat: ancestry should not be confused with dyadic dilation used at defining the sets  $\Gamma_k$. Set $I(x,y)$ to denote the smallest dyadic interval containing both $x$ and $y$. Then $\abs{I(x,y)}=\delta(x,y)$. For every $I\in\mathcal{D}$ with $\abs{I}<\abs{I(x,y)}$ we have $h_I(x)=0$ or $h_I(y)=0$ hence the sum in the definition of $\Omega_m(x,y)$ is performed only on $I(x,y)$ and its ancestors. Thus
	\begin{align}
	\Omega_m(x,y) &= \delta(x,y)\left[-m(I(x,y))\frac{1}{\abs{I(x,y)}}+ \sum_{j\geq 1} m(I^{(j)}(x,y))\frac{1}{\abs{I^{(j)}(x,y)}}\right]\notag\\
	&= -m(I(x,y)) + \sum_{j\geq 1} 2^{-j} m(I^{(j)}(x,y)).\label{eq:Omegamexpression}
	\end{align}
	Notice first that if $(\xi,\eta)\in (\mathbb{R}^+)^2$ is such that $(\xi,\eta)\in B(I(x,y))$ then $I(\xi,\eta)=I(x,y)$ and also of course $I^{(j)}(x,y)=I^{(j)}(\xi,\eta)$. Hence $\Omega_m$ is $\mathscr{F}$ measurable. In order to check the dyadic homogeneity of degree zero it is enough to prove that $\Omega(2^lx,2^l y)=\Omega(x,y)$ for every $l\in\mathbb{Z}$. It suffices to prove the identity for $l=1$. Now, from the fact that $m(I)=m(2I)$,
	\begin{align*}
	\Omega_m(2x,2y) &= -m(I(2x,2y)) + \sum_{j\geq 1} 2^{-j} m(I^{(j)}(2x,2y))\\
	&= -m(2I(x,y)) + \sum_{j\geq 1} 2^{-j} m([2I(x,y)]^{(j)})\\
	&= -m(I(x,y)) + \sum_{j\geq 1} 2^{-j} m(I^{(j)}(x,y)),
	\end{align*}
	since duplication and ancestry commute and the smallest dyadic interval containing $2x$ and $2y$ is $2I(x,y)$.
\end{proof}

\section{Dyadic fractional differentiation. Fractional Laplacian and fractional gradient}\label{sec:LaplacianGradient}

The small fractional powers $\sigma$ of the standard Laplacian in $\mathbb{R}^n$ are defined on smooth functions as the absolutely convergent integrals of the form
\begin{equation*}
(-\Delta)^{\sigma} f(x) = \int_{\mathbb{R}^n} \frac{f(y)-f(x)}{\abs{x-y}^{n+2\sigma}} dy.
\end{equation*}
See for example \cite{CaSi2007}. Since the space of homogeneous type $(\mathbb{R}^+, \delta, \textrm{Lebesgue measure})$ is $1$-Ahlfors (or normal), for $f$ bounded and $\delta$-Lipschitz, we define
\begin{equation*}
D^{s} f(x) = \int_{\mathbb{R}^+} \frac{f(y)-f(x)}{\delta(x,y)^{1+s}} dy.
\end{equation*}
See \cite{AiBoGo13}. The spectral analysis of $D^s$, $0<s<1$, is provided naturally the Haar system,
\begin{equation*}
D^s f = \sum_{h\in\mathscr{H}} \abs{I(h)}^{-s}\proin{f}{h} h.
\end{equation*}
Even when we do not have a second order Laplacian in this setting, we shall keep writing $(-\Delta)^{s/2}_{dy}$ instead of $D^s$ in order to emphasize its central role as the main elliptic type operator of the dyadic setting.

For $m\in\mathscr{M}$ and $f\in \mathcal{S}(\mathscr{H})$, the linear span of $\mathscr{H}$, the operator
\begin{equation*}
D^s_m f(x)=\sum_{h\in\mathscr{H}} m(h)\abs{I(h)}^{-s}\proin{f}{h}h(x)
\end{equation*}
is called the dyadic directional derivative of $f$ of order $s$ in the direction of $m$.

A particular case of $m$ will provide the partial derivatives and the gradient. Take $i\geq 0$ and $m_i: \mathcal{D}\to\mathbb{R}$ dyadic homogeneous of degree zero, such that when $m_i$  is restricted to $\mathcal{D}^0$ coincides with the vector $\mathbf{e}_i$ of the canonic basis. More precisely, $m_i([k,k+1))=\delta_{ik}$ and $m_i$ extends to $\mathcal{D}$ by dyadic homogeneity of order zero. We write $D^s_{(i)}$ for $D^s_{m_i}$ and we call it the $i$-th partial dyadic differential operator of order $s$. The explicit spectral formula for $D^s_{(i)}$ is given by
\begin{equation*}
D^s_{(i)} f(x) = \sum_{j\in\mathbb{Z}} 2^{sj} \proin{f}{h^j_i}h^j_i(x).
\end{equation*}
The gradient of $f$ is the infinite dimensional vector
\begin{equation*}
\boldsymbol\nabla^s_{dy} f = \left(D^s_{(i)} f:\, i\geq 0\right).
\end{equation*}
A basic reason to consider this operator as a good version in the current setting of the standard gradient, is its capability to measure the energy contained in a signal $f$ defined in $\mathbb{R}^+$. The next result will clarify and precise this assertion.

The natural energy form for the dyadic setting is given by
\begin{equation*}
\mathscr{E}^\delta_s (f) = \iint_{\mathbb{R}^+\times \mathbb{R}^+}\abs{\frac{f(x)-f(y)}{\delta(x,y)^s}}^ 2 \frac{dx dy}{\delta(x,y)}.
\end{equation*}
In \cite{AiBolGo19}, Corollary~3.4 it is proved that for $f\in\mathcal{S}(\mathscr{H})$,
\begin{equation*}
\mathscr{E}^\delta_s (f) = c \sum_{h\in\mathscr{H}}\frac{\abs{\proin{f}{h}}^2}{\abs{I(h)}^{2s}},
\end{equation*}
for some positive constant $c$.

\begin{proposition}
	The $s$-dyadic energy of $f$ is finite if and only if $\boldsymbol\nabla^s_{dy} f\in L^2(\mathbb{R}^+,\ell^2(\mathbb{Z}^+))$. Moreover
	\begin{equation*}
	\mathscr{E}^\delta_s (f) = c \norm{\abs{\mathbf{\nabla}^s_{dy} f}_{\ell^2}}^2_{L^2} = c \sum_{i\geq 0}\norm{D^s_{(i)}f}^2_2
	\end{equation*}
\end{proposition}
\begin{proof}
	Notice first $\norm{D^s_{(i)} f}^2_2 = \sum_{j\in\mathbb{Z}} 2^{2sj}\abs{\proin{f}{h^j_i}}^2$. Hence
	\begin{align*}
	\sum_{i\geq 0}\norm{D^s_{(i)} f}^2_2 &= \sum_{i\geq 0}\sum_{j\in\mathbb{Z}}2^{2sj}\abs{\proin{f}{h^j_i}}^2\\
	&= \sum_{h\in\mathscr{H}}\frac{\abs{\proin{f}{h}}^2}{\abs{I(h)}^{2s}}\\
	&=\frac{1}{c}\iint_{\mathbb{R}^+\times \mathbb{R}^+}\abs{\frac{f(x)-f(y)}{\delta(x,y)^s}}^2\frac{dx dy}{\delta(x,y)}\\
	&=\frac{1}{c}\mathscr{E}^\delta_s (f).
	\end{align*}
\end{proof}
By polarization, the above result shows that the underlying bilinear form takes a familiar aspect in terms of the gradient. Precisely,
\begin{equation*}
c\int_{\mathbb{R}^+}\proin{\boldsymbol\nabla^s_{dy} f}{\boldsymbol\nabla^s_{dy}g} dy = \iint_{\mathbb{R}^+\times\mathbb{R}^+}\frac{(f(x)-f(y))}{\delta^s(x,y)}\frac{(g(x)-g(y))}{\delta^s(x,y)}\frac{dx dy}{\delta(x,y)}.
\end{equation*}

Let us finish this section with a proposition regarding the symbolic calculus of the above differential operators in terms of our central operator $(-\Delta)^{s/2}_{dy}$.
\begin{proposition}
	\quad
	\begin{enumerate}
		\item With $m\in\mathscr{M}$ and $T_m g = \sum_{h\in\mathscr{H}} m(h) \proin{g}{h} h$,
		\begin{equation*}
		D^s_m = T_m\circ (-\Delta)^{s/2}_{dy};
		\end{equation*}
		\item with $i\geq 0$ and $T_{(i)}g = \sum_{j\in\mathbb{Z}}\proin{g}{h^j_i} h^j_i$,
		\begin{equation*}
		D^s_{(i)}  = T_{(i)}\circ (-\Delta)^{s/2}_{dy};
		\end{equation*}
		\item with $\mathbf{T} g = \left(\sum_{j\in\mathbb{Z}}\proin{g}{h^j_i} h^j_i:\, i\geq 0\right)$,
		\begin{equation*}
		\boldsymbol\nabla^s_{dy}  = \mathbf{T}\circ (-\Delta)^{s/2}_{dy}.
		\end{equation*}
		\end{enumerate}
\end{proposition}

\begin{proof}
Is straightforward for functions in $\mathcal{S}(\mathscr{H})$ and the definition of the operators in terms of their multipliers.
\end{proof}

\section{On $\ell^2(\mathbb{Z}^+)$ valued Calder\'on-Zygmund singular integrals on spaces of homogeneous type}\label{sec:CalderonZygmundSingularETH}

Even when the main result of this section can be obtained in more general situations, we shall only consider the basic operators needed in Section~\ref{sec:SobolevRegularity}
in order to obtain the Sobolev regularity results for solutions of $(-\Delta)^{s/2}_{dy} u=f$.
The extension of the theory of Calder\'on-Zygmund to vector valued settings is a classical result (see \cite{BeCaPa62} and \cite{RdFRuTo86}). On the other hand, the extension of the Calder\'on-Zygmund theory (see \cite{CaZyActa52} and \cite{CaZy78}) to general domains, such as spaces of homogeneous type, is also well known (see \cite{Aimar}, \cite{Aimar85}, \cite{AiGoPetermichl}). In \cite{GraLiDa09} the two settings are considered at once with weaker assumptions regarding the regularity of the kernel.

A space $(X,d,\mu)$ with $d$ a metric on $X$ is said to be a $1$-Ahlfors or normal space of homogeneous type if $\mu$ is a Borel measure on $X$ such that $\mu(B(x,r))\simeq r$ uniformly in $x\in X$. In other words, there exist constants $0<c_1\leq c_2<\infty$ such that $c_1 r\leq \mu(B(x,r))\leq c_2 r$ for every $x\in X$ and every $r>0$.

For simplicity we denote by $\ell^2$ the space $\ell^2(\mathbb{Z}^+)$ with $\mathbb{Z}^+$ the set of nonnegative integers.
\begin{definition}\label{def:CZoperator}
	Let $(X,d,\mu)$ be a $1$-Ahlfors metric spaces. Let $\mathbf{T}: L^2(\mathbb{R}^+)\to L^2(\mathbb{R}^+,\ell^2)$ be a bounded and linear operator from the real valued $\varphi\in L^2(\mathbb{R}^+)$ functions into the $\ell^2$ valued functions with the norm $\norm{\abs{\mathbf{T}\varphi}_{\ell^2}}_{L^2(\mathbb{R}^+)}$. We say that $\mathbf{T}$ is an $\ell^2$-valued Calder\'on-Zygmund operator on $(X,d,\mu)$ if there exists a kernel $\mathbf{K}\in L^1_{loc}(X\times X\setminus \triangle, \ell^2)$ such that
	\begin{enumerate}[(i)]
		\item there exists $C_0>0$ with
	%	\begin{equation*}
		$\abs{\mathbf{K}(x,y)}_{\ell^2}\leq \frac{C_0}{d(x,y)}$, $x\neq y$;
	%	\end{equation*}
		\item there exists $C_1>0$ and $\gamma >0$ such that\\
		ii.a) $\abs{\mathbf{K}(x',y)- \mathbf{K}(x,y)}_{\ell^2}\leq C_1 \frac{d(x',x)}{d(x,y)^{1+\gamma}}$ when $2d(x',x)\leq d(x,y)$;\\
		ii.b) $\abs{\mathbf{K}(x,y')- \mathbf{K}(x,y)}_{\ell^2}\leq C_1 \frac{d(y,y')}{d(x,y)^{1+\gamma}}$ when $2d(y,y')\leq d(x,y)$;
		\item there  exist a dense subspace $S$ of $L^2(\mathbb{R}^+)$ and a  dense subspace $\boldsymbol S$ of $L^2(\mathbb{R}^+,\ell^2)$ such that the formula
		\begin{equation*}
		\proin{\mathbf{T}\varphi}{\boldsymbol\psi} = \iint_{X\times X} \proin{\mathbf{K}(x,y) \varphi(y)}{\boldsymbol\psi(x)}_{\ell^2} dx dy
		\end{equation*}
hold for every $\varphi\in S$ and every $\boldsymbol\psi\in\boldsymbol S$ with $\delta(\supp\varphi,\supp\boldsymbol\psi)>0$.
	\end{enumerate}
\end{definition}
\begin{theorem}
	Let $\mathbf{T}$ be an $\ell^2$-valued Calder\'on-Zygmund operator defined in $(X,d,\mu)$. Then $\mathbf{T}$ is bounded as an operator from $L^p(\mathbb{R}^+)$ with values in $L^p(\mathbb{R}^+,\ell^2)$ for every $1<p<\infty$. Precisely
	\begin{equation*}
	\norm{\abs{\mathbf{T}g}_{\ell^2}}_{L^p(\mathbb{R}^+)}\leq C_p \norm{g}_{L^p(\mathbb{R}^+)}
	\end{equation*}
	for some constant $C_p$ and every $g\in L^p(\mathbb{R}^+)$.
\end{theorem}
Since the space $\mathcal{L}(\ell^2,\mathbb{R})$ of all linear and bounded operators from $\ell^2$ with values in $\mathbb{R}$ can be identified with $\ell^2$ itself, the above theorem follows from the results in \cite{GraLiDa09}, for example.

\section{Sobolev regularity of solutions of $(-\Delta)^{s/2}_{dy} u = f$} \label{sec:SobolevRegularity}

As we proved in Section~\ref{sec:LaplacianGradient}, $\boldsymbol\nabla^s_{dy}=\mathbf{T}\circ (-\Delta)^{s/2}_{dy}$, where $\mathbf{T}g=\left(\sum_{j\in\mathbb{Z}}\proin{g}{h^j_i} h^j_i: i\geq 0\right)$. Hence our result will be a consequence of the result in Section~\ref{sec:CalderonZygmundSingularETH} provided we show that $\mathbf{T}$ is an $\ell^2$-valued Calder\'on-Zygmund operator in $(\mathbb{R}^+, \delta, \textrm{Lebesgue measure})$.
\begin{theorem}
	$\mathbf{T}$ is an $\ell^2$-valued Calder\'on-Zygmund operator  in $(\mathbb{R}^+, \delta, \textrm{Lebesgue measure})$ and
	\begin{equation*}
	\norm{\abs{\boldsymbol\nabla^s_{dy} u}_{\ell^2}}_{L^p(\mathbb{R}^+)}\leq C_p\norm{(-\Delta)^{s/2}_{dy}u}_{L^p(\mathbb{R}^+)},
	\end{equation*}
	for $1<p<\infty$.
\end{theorem}
\begin{proof}
All we need is to check that $\mathbf{T}$ satisfies the Definition~\ref{def:CZoperator} of the previous section with $X=\mathbb{R}^+$, $d=\delta$ and $d\mu=dx$. The $L^2(\mathbb{R}^+,\ell^2)$ boundedness of $\mathbf{T}$ is a straightforward consequence of the $L^2(\mathbb{R}^+)$ summability of the Haar series of an $L^2(\mathbb{R}^+)$ function. Let us start by proving the estimate
\begin{equation}\label{eq:estimateabovefordeltaquadratic}
\sum_{i\geq 0}\left(\sum_{j\in\mathbb{Z}} h^j_i(x) h^j_i(y)\right)^2 \leq\frac{4}{\delta^2(x,y)}
\end{equation}
 for every $x\neq y$ in $\mathbb{R}^+$. This estimate will be useful at proving $(i)$ in Definition~\ref{def:CZoperator} and helpful in order to show the absolute convergence of some of the multiple integrals needed to show $(iii)$. Set $K_{(i)}(x,y)=\sum_{j\in\mathbb{Z}}h^j_i(x) h^j_i(y)$, $x\neq y$, $\boldsymbol K = (K_{(i)}: i\geq 0)$, then we have to estimate the $\ell^2$ norm of $\mathbf{K}(x,y)$ for $x\neq y$, $\abs{\mathbf{K}(x,y)}^2_{\ell^2}=\sum_{i\geq 0}\abs{K_{(i)}(x,y)}^2$. Now for $i\geq 0$,
\begin{equation*}
K_{(i)} (x,y) = \sum_{j\in\mathbb{Z}} h^j_i(x) h^j_i(y) = \frac{\Omega_{(i)}(x,y)}{\delta(x,y)}
\end{equation*}
and, from the formula for $\Omega_{(i)}=\Omega_{m_i}$ given by \eqref{eq:Omegamexpression}, we have
\begin{equation*}
K_{(i)}(x,y) = \frac{1}{\delta(x,y)}\left[-m_i(I(x,y)) + \sum_{j\geq 1} 2^{-j} m_i(I^{(j)}(x,y))\right].
\end{equation*}
Hence, for $x\neq y$, we have
\begin{equation*}
\abs{\mathbf{K}(x,y)}_{\ell^2} \leq \frac{1}{\delta(x,y)}\left[\sqrt{\sum_{i\geq 0} m_i^2(I(x,y))} + \sum_{j\geq 1} 2^{-j}\sqrt{\sum_{i\geq 0}m_i^2(I^{(j)}(x,y))}\right]
\leq \frac{2}{\delta(x,y)},
\end{equation*}
because of the definition of $m_i$.

Let us formally check that properties $(ii.a)$ and $(ii.b)$ hold with $d=\delta$. Take three points, $x$, $x'$ and $y\in\mathbb{R}^+$ with $2\delta(x,x')\leq\delta(x,y)$. This means, with the above notation, that $2\abs{I(x,x')}\leq\abs{I(x,y)}$ and, since $x\in I(x',x)\cap I(x,y)$ and both are dyadic intervals, we have $I(x,y)\supsetneq I(x,x')$. Notice also that $y\notin I(x,x')$, hence $I(x,y)=I(x',y)$, $\delta(x,y)=\delta(x',y)$ and $I^{(j)}(x,y)= I^{(j)}(x',y)$ for every $j\geq 1$.
So that
\begin{align*}
K_{(i)}(x',y) &= \frac{\Omega_{(i)}(x',y)}{\delta(x',y)}\\
&=\frac{1}{\delta(x',y)}\left(-m(I(x',y)) + \sum_{j\geq 1} 2^{-j} m_i(I^{(j)}(x',y))\right)\\
&= \frac{\Omega_{(i)}(x,y)}{\delta(x,y)}\\
&= K_{(i)}(x,y),
\end{align*}
and $\mathbf{K}(x',y)=\mathbf{K}(x,y)$ for $2\delta(x,x')\leq \delta(x,y)$ and condition $(ii.a)$ holds trivially.

Let us finally prove $(iii)$ in Definition~\ref{def:CZoperator}. It suffices to show $(iii)$ for $\varphi\in S(\mathscr{H})$, the scalar linear span of $\mathscr{H}$, and $\boldsymbol\psi\in \boldsymbol S(\mathscr{H})$, the vector linear span of $\mathscr{H}$. Since we are only dealing with finite sums we see that
	\begin{align*}
	\proin{\mathbf{T}\varphi}{\boldsymbol\psi} &= \sum_{i\geq 0} \proin{T_i\varphi}{\psi_i}
	\\
	&= \sum_{i\geq 0} \int_{x\in\mathbb{R}^+} T_i\varphi(x) \psi_i(x) dx\\
	&= \sum_{i\geq 0} \int_{x\in\mathbb{R}^+}\sum_{j\in\mathbb{Z}}\proin{\varphi}{h^j_i} h^j_i(x)\psi_i(x) dx\\
    &= \sum_{i\geq 0} \int_{x\in\mathbb{R}^+}\left[\sum_{j\in\mathbb{Z}}\left(\int_{\mathbb{R}^+}\varphi(y) h^j_i(y) dy\right)h^j_i(x)\right]\psi_i(x) dx.
	\end{align*}
Now set $A=\supp\varphi$ and $B=\supp\boldsymbol\psi$. We are assuming $A\cap B=\emptyset$, hence $\delta(A,B)>0$. Since, from \eqref{eq:estimateabovefordeltaquadratic} we have
\begin{align*}
\int_{x\in B}\int_{y\in A}\sum_{i\geq 0} \abs{\sum_{j\in\mathbb{Z}} h^j_i(y) h^j_i(x)} & \abs{\varphi(y)}\abs{\psi_i(x)} dy dx \\
&\leq 2 \int_{y\in A}\abs{\varphi(y)}\left(\int_{x\in B}\frac{dx}{\delta^2(x,y)}\right)^{1/2} dy \norm{\abs{\boldsymbol\psi}_{\ell^2}}_{L^2}
\end{align*}
which is finite, then
	\begin{equation*}
	\proin{\mathbf{T}\varphi}{\boldsymbol\psi} = \int_{x\in\mathbb{R}^+}\int_{y\in\mathbb{R}^+} \proin{\mathbf{K}(x,y)\varphi(y)}{\boldsymbol\psi(x)}_{\ell^2} dy dx,
	\end{equation*}
with $K_{(i)}(x,y) = \sum_{j\in\mathbb{Z}} h^j_i(x) h^j_i(y)$ and $\boldsymbol K = (K_{(i)}: i\geq 0)$, as desired.
\end{proof}

%%%%%%%%%%%%%%%% REFERENCES %%%%%%%%%%%%%%%%%%%%%%%%%%%%%%%%%%%%%%%%%%%%%%%%%%%%%%%%

%\bibliographystyle{amsalpha}
%\bibliography{ref}

\providecommand{\bysame}{\leavevmode\hbox to3em{\hrulefill}\thinspace}
\providecommand{\MR}{\relax\ifhmode\unskip\space\fi MR }
% \MRhref is called by the amsart/book/proc definition of \MR.
\providecommand{\MRhref}[2]{%
  \href{http://www.ams.org/mathscinet-getitem?mr=#1}{#2}
}
\providecommand{\href}[2]{#2}

%%%%%%%%%%%%%%%% Address Authors %%%%%%%%%%%%%%%%%%%%%%%%%%%%%%%%%%%%%%%%%%%%%%%%%%%

%\bigskip

\bigskip
\noindent{\footnotesize Hugo Aimar, Juan Comesatti, and Ivana G\'omez.
	\textsc{Instituto de Matem\'{a}tica Aplicada del Litoral, UNL, CONICET.}
	
	%
	%\smallskip
%	\noindent \textit{Address.} \textmd{CCT CONICET Santa Fe, Predio ``Alberto Cassano'', Colectora Ruta Nac.~168 km 0, Paraje El Pozo, S3007ABA Santa Fe, Argentina.}
%}
%
\medskip

\noindent{\footnotesize Luis Nowak.
	\textsc{Instituto de Investigaci\'on en Tecnolog\'ias y Ciencias de la Ingenier\'ia, Universidad Nacional del Comahue, CONICET; Departamento de Matem\'atica-FaEA-UNComa.}
	
	\noindent\textmd{Neuqu\'en, Argentina}
}

\medskip
\noindent \textit{Corresponding author. E-mail:} \verb|ivanagomez@santafe-conicet.gov.ar|

\noindent \textit{Address.} \textmd{IMAL, CCT CONICET Santa Fe, Predio ``Alberto Cassano'', Colectora Ruta Nac.~168 km 0, Paraje El Pozo, S3007ABA Santa Fe, Argentina.}
}

\end{document}